\documentclass[11pt]{article}
\usepackage[english]{babel}
\usepackage[T1]{fontenc}
\usepackage[utf8]{inputenc}
\usepackage{amsmath}
\usepackage{amssymb}
\usepackage{amsfonts}
\usepackage{amsthm}
\usepackage{bm}
\usepackage{mathtools}
\usepackage{ascmac}
\usepackage{framed}
\usepackage{ulem}

\usepackage{color}								
\definecolor{lightgray}{rgb}{0.75,0.75,0.75}

 {\endMakeFramed}

\newcommand{\eps}{\varepsilon}

\newcommand{\R}{\mathbb{R}}
\renewcommand{\S}{\mathbb{S}}

\newcommand{\boD}{\mathcal{D}}

\newcommand{\boH}{\mathcal{H}}
\newcommand{\boI}{\mathcal{I}}

\newcommand{\boL}{\mathcal{L}}
\newcommand{\boM}{\mathcal{M}}

\newcommand{\la}{\ensuremath{\lambda}}
\newcommand{\La}{\ensuremath{\Lambda}}
\newcommand{\te}{\ensuremath{\theta}}
\newcommand{\al}{\ensuremath{\alpha}}
\newcommand{\be}{\ensuremath{\beta}}

\newcommand{\loc}{\mathrm{loc}}

\newtheorem*{claim*}{Claim}

\newtheorem*{cor*}{Corollary}
\newtheorem{lem}{Lemma}
\newtheorem{prop}{Proposition}

\newtheorem{thm}{Theorem}

\theoremstyle{definition}

\newtheorem{remark}{Remark}

\theoremstyle{remark}

\makeatletter												%
  \@addtoreset{equation}{section}						%
\makeatother

\newcommand{\nor}[2]{\left\| {#1} \right\|_{#2}}		
		
\newcommand{\inp}[2]{\langle {#1} , {#2} \rangle }	

\newcommand{\Del}{{\Delta}}							
\newcommand{\rd}{{\partial}}								
\newcommand{\nab}{{\nabla}}							

\newcommand{\bb}{{\mathbf{b}}}
\newcommand{\bh}{{\mathbf{h}}}
\newcommand{\bn}{{\mathbf{n}}}
\newcommand{\boe}{{\mathbf{e}}}
\newcommand{\bom}{{\mathbf{m}}}
\newcommand{\bphi}{{\bm{\phi}}}

\newcommand{\bxi}{{\bm{\xi}}}
\newcommand{\bJ}{{\mathbf{J}}}

\begin{document}

\title{
Phase transition threshold and stability of magnetic skyrmions
}
\author{Slim Ibrahim, Ikkei Shimizu}
\date{}
\maketitle

\begin{abstract}
We examine the stability of 
vortex-like configuration of magnetization in magnetic materials, so-called the magnetic skyrmion. 
These correspond to critical points
of the Landau-Lifshitz energy with the Dzyaloshinskii-Moriya (DM) interactions. 
In an earlier work of D\"oring and Melcher, it is known that the skyrmion is a ground state when the coefficient of the DM term is small. 
In this paper, we prove that there is an explicit critical value of the coefficient above which the skyrmion is unstable, while stable below this threshold. 
Moreover, we show that in the unstable regime, the infimum of energy is not bounded below, by giving an explicit counterexample with a sort of helical configuration. 
This mathematically explains the occurrence of phase transition observed in some experiments. 
\end{abstract}

\section{Introduction}\label{S1}

We consider the Landau-Lifshitz energy functional of the form 
\begin{equation}\label{e1}
E_p [\bn] = D [\bn] + r H [\bn] 
 + V_p [\bn],\qquad \bn:\R^2\to\S^2
\end{equation}
where $r>0$, $p\ge 2$ are constants, and 
$$
D[\bn] := \frac 12 \int_{\R^2} |\nab \bn|^2 dx;\qquad \text{Dirichlet energy},
$$
$$
H[\bn] := \int_{\R^2} (\bn -\boe_3) \cdot \nab \times \bn dx;\qquad \text{helicity},
$$
$$
V_p [\bn] := \frac{1}{2^{p-1}} \int_{\R^2} |\bn - \boe_3|^p dx = \frac{1}{2^{p/2-1}} \int_{\R^2} (1-n_3)^{p/2} dx;\qquad \text{potential energy}.
$$
First we point out that $\nab \times \bn$ can be defined as 
$$
\nab \times \bn = 
\begin{pmatrix}
\rd_1 \\ \rd_2 \\ 0
\end{pmatrix}
\times 
\bn,
$$
which is the there-dimensional curl acting on maps depending only on $x_1$ and $x_2$. 
The Landau-Lifshitz energy arises in micromagnetics, 
where $\bn$ represents magnetization vector in a magnetic material. 
The equilibrium state of magnetization is characterized by the stable critical points of $E_p$ (see \cite{HS} for the general theory). 
In the energy $E_p$, $D$ represents the exchange interaction, while $V_p$ arises from 
the external field and crystalline structure which produces 
the easy-axis anisotropy perpendicular to the planar material. 
$H[\bn]$ stands for the Dzyaloshinskii-Moriya interaction of magnetization, 
which emerges in some particular crystalline structure \cite{BogYab89, BogHub94}. 
When this kind of magnetic material is analyzed 
under the effect of strong external field, 
then localized vortex-like configuration of magnetization, called \textit{magnetic skyrmion}, appears. 
This has actually been observed in experiments \cite{Ye12, NagTok13}, and theoretically justified in \cite{Mel14, DorMel17, LiMel18, GusWan21}. 
On the other hand, when the effect of the external field is weak, then 
some helical shape different from skyrmion state are observed \cite{Ye12, NagTok13}. 
This suggests the occurrence of a phase transition of the equilibrium state. 
There is, however, no result on the rigorous study justifying the above phenomena so far. 
The purpose of this paper is to establish the rigorous proof of phase transition through the analysis of the Landau-Lifshitz energy.\par
In $E_p$, the parameter $r$ measures the strength of external fields; 
larger values of $r$ correspond to weaker field strength. 
Before going any further, let us explain that putting different weights on $H$ and $V$, 
the energy can be formulated through two parameters as follows:
$$
D [\bn] + \kappa H [\bn] 
 + \mu V_p [\bn].
$$
However, one can always normalize the coefficient of $V_p$ using the rescaling $\bn(x)\mapsto \bn(\sqrt{\mu} x)$, which reduces the problem to \eqref{e1} with $r=\kappa \mu^{-1/2}$.\par
In the present paper, we are especially interested in the case $p=4$, where the 
energy has some special structure of a nice factorization, discussed later in details. 
Define the function space of maps $\bn$ by
$$
\boM_4 := \{
 \bn :\R^2\to \S^2\ :\ D[\bn] + V_4 [\bn] <\infty
\}
$$
endowed with the metric $d(\bn, \bom) : = \nor{\bn-\bom}{L^4(\R^2)} + \nor{\nab (\bn-\bom)}{L^2(\R^2)}$ for $\bn,\bom\in \boM_4$. 
Recall that the energy functional $E_4$ is well-defined on $\boM_4$ (see \cite[Page 7]{DorMel17} for the proof). 
In addition, 
thanks to the well-known inequality of Wente \cite{Wen69}, the topological degree
$$
Q [\bn] := \frac 1{4\pi} \int_{\R^2} \bn\cdot \rd_1 \bn \times \rd_2 \bn dx
$$
is also a well-defined, integer-valued functional on $\boM_4$. 
$Q[\bn]$ represents the total number of skyrmions, 
and its sign gives an idea about their directions of rotation. 
In this paper, we will restrict ourselves to the case $Q=-1$ 
corresponding to a single skyrmion, which is known to be the most stable homotopy class (see \cite{Mel14}). 
Moreover, solutions in this class are known to enjoy additional properties.\par
%
%
%
The critical points of $E_4$ satisfy the Euler-Lagrange equation
\begin{equation}\label{e2}
-\Del \bn + 2r \nab \times \bn - (1-n_3) \boe_3 - \La (\bn) \bn =0
\end{equation}
where
$$
\La (\bn) := |\nab \bn|^2 + 2r \bn \cdot (\nab \times \bn) - (1-n_3) n_3.
$$
Then \eqref{e2} has an explicit solution:
\begin{equation}\label{e3}
\bh^{2r} (x) := \bh \left( \frac{x}{2r} \right),\qquad 
\bh (x) := \left(
\frac{-2x_2}{1+|x|^2} , \frac{2x_1}{1+|x|^2} , -\frac{1-|x|^2}{1+|x|^2}
\right).
\end{equation}
Note that $\bh^{2r}$ is a harmonic map, with $\bh^{2r}\in \boM_4$ and $Q[\bh^{2r}]=-1$. 
In the work of D\"oring and Melcher \cite{DorMel17}, it is shown that 
when $0<r\le 1$, $\bh^{2r}$ is a global minimizer of the energy;
$$
\min_{\substack{\bn\in\boM_4 \\ Q[\bn] =-1}} E_4 [\bn] = E_4 [\bh^{2r}] = 4\pi (1-2r^2).
$$
In particular, their result explains the formation of skyrmion since the $\bh^{2r}$ has the desired configuration. 
The essence of this result is that $E_4$ can be factorized as
\begin{equation}\label{e4}
E_4 [\bn] - 4\pi r^2 Q[\bn] = \frac {r^2}2 \int_{\R^2} |\boD_1^{r} \bn + \bn\times \boD^{r}_2 \bn|^2 dx + (1- r^2) D[\bn]
\end{equation}
where the helical derivatives $\boD_j^r$ are defined by
$$
\boD_j^r
:= \rd_j - \frac {1}{r} \boe_j \times \cdot.
$$
When $r\le 1$, then the minimality immediately follows from \eqref{e4} because $\bh^{2r}$ is a minimizer of $D$ with $Q=-1$, and satisfies 
$\boD_1^{r} \bh^{2r} + \bn\times \boD^{r}_2 \bh^{2r} =0$. 
However, this argument clearly breaks down when $r>1$, where even the stability of $\bh^{2r}$ has not been understood so far. \par
Our main result is the instability of $E_4$ around $\bh^{2r}$ when $r>1$.

\begin{thm}\label{T1}
If $r>1$, then the critical point $\bh^{2r}$ of the energy $E_4$ is unstable, in the sense that 
for any neighborhood of $\bh^{2r}$ there exists $\bn\in \boM_4$ such that $E_4[\bn] - E_4[\bh^{2r}] <0$.
\end{thm}

Theorem \ref{T1} rigorously explains why 
the phase transition occurs;
the stability of skyrmions ceases to hold when the external field is weak. 
Moreover, Theorem \ref{T1}, together with D\"oring-Melcher's results, explicitly quantifies the threshold of phase transition at $r=1$. \par
For the proof of Theorem \ref{T1}, we follow the framework of \cite{LiMel18}; 
We rewrite the quadratic form of Hessian in terms of the coordinates of moving frame, and then decompose it with respect to Fourier modes of argument variable. 
To unveil unstable factors, we apply a rescaling argument by following the strategy of \cite{LamZun22}. 
Then we can find negative directions of Hessian in the Fourier modes higher than or equal to $2$. 
It is quite worth to note that as long as this strategy, the criticality of $r=1$ is shown through the third Fourier mode; namely the unstable mode can only be observed via Hessian at third Fourier mode when $r$ is close to $1+$. 
We also observe that the Hessian at $0$-th and first Fourier modes is always positive definite, regardless of $r$. This mechanism has already been observed 
in that of the Ginzburg-Landau energy \cite{LamZun22}.\par
%
Next, it is natural to ask about 
the existence of minimizer of $E_4$ in the unstable regime. We show that if $r>1$, then the energy is unbounded from below.
\begin{thm}\label{T2}
If $r>1$, then
$$
\min_{\substack{\bn\in\boM_4 \\ Q[\bn] =-1}} E_4 [\bn] = -\infty.
$$
\end{thm}
This theorem suggests that \eqref{e1} on $\boM_4$ is not well-suited to characterize 
equilibrium states in the regime of weak field.
Nevertheless, it is worth mentioning that our example of the unbounded sequence 
has the same helical structure as observed in experiments of \cite{Ye12}. 
Moreover, our construction is by stretching a skyrmion in one direction, which roughly gives us the information of how the instability of skyrmions occurs. 
Note that this cannot happen when the domain is bounded, as minimizers may exist as expected by experiments. This is beyond the scope of this paper.
\par
We conclude this introduction by mentioning some known related studies. 
The minimizing problem of $E_p$ is first addressed by \cite{Mel14} when $p=2$, and the analysis is extended to various settings in \cite{DorMel17, LiMel18, GusWan21}. 
Recently, the geometric interpretation of the integrand of the first term in the left hand side of \eqref{e4} is given in \cite{Sch19, BarRosSch20}, which yields a family of formal solutions to the corresponding Bogomol'nyi-type equation. 
The dynamical equation corresponding to the energy related to \eqref{e1} is also investigated with the Gilbert damping by \cite{DorMel17}, or without damping by the second author \cite{S}. 
It is worth noting that in the latter case, the equation is closely related to the nonlinear Schr\"odinger equation, and in fact, when the energy only consists of $D[\bn]$, then a sort of dispersive properties are observed \cite{GKT2, GNT, BejIonKenTat11, BejIonKenTat13}.\par
The organization of this paper is as follows. In Section \ref{S2}, we first derive the Hessian of $E_4$, then reduce the problem to its analysis. 
The main part is Section \ref{S3} where we construct an unstable direction of the Hessian, which concludes Theorem \ref{T1}. In Section \ref{S4}, we prove Theorem \ref{T2} by constructing a sequence which gives infinitely negative energy. 
In Section \ref{S5}, we prove the technical lemmas used in the main argument.

%
%

\section{Hessian}\label{S2}

First of all, we observe that the difference of energy from $\bh^{2r}$ can be written as a quadratic form.

\begin{lem}\label{L1}
Let $\bn\in \boM_4$ and $\bxi := \bn- \bh^{2r}$. Then, 
\begin{equation}\label{e2.1}
E_4[\bn] - E_4[\bh^{2r}] = \frac 12 \inp{\boL \bxi}{\bxi}_{L^2}
\end{equation}
where
$$
\boL \bxi := -\Del \bxi +2r\nab\times \bxi + \xi_3\boe_3 -\La (\bh^{2r}) \bxi.
$$
\end{lem}

\begin{proof}
By the criticality of $\bh^{2r}$ for $E_4$, we have
$$
E_4[\bn] - E_4[\bh^{2r}] = 
\int \La(\bh^{2r}) \bh^{2r}\cdot \bxi
-\frac 12 \int \Del \bxi\cdot \bxi + r \int\nab\times\bxi \cdot \bxi 
+ \frac 12 \int \xi_3^2.
$$
The constraint $|\bn|=|\bh^{2r}|=1$ yields
$$
2 \bxi\cdot \bh^{2r} + |\bxi|^2 = 0.
$$
Thus
$$
\int \La(\bh^{2r}) \bh^{2r} \cdot \bxi = -\frac 12 \int \La (\bh^{2r}) |\bxi|^2,
$$
which completes the proof.
\end{proof}

Let us focus on the quadratic form defined by the right hand side of \eqref{e2.1}. 
Now we claim that 
the perturbation $\bphi$ may be linearized into the tangent space.

\begin{prop}[Reduction of the theorem]\label{P2.1}
Suppose that there is $\bphi\in H^1(\R^2)$ with $\bphi \cdot \bh^{2r}=0$ such that 
$\inp{\boL \bphi}{\bphi} <0$, then for any neighborhood of $\bh^{2r}$, there exists $\bn\in \boM_4$ such that $E_4 [\bn] - E_4[\bh^{2r}]<0$. 
\end{prop}

\begin{proof}
For $t>0$, let
$\bn_t := \frac{\bh^{2r} + t\bphi}{|\bh^{2r} + t\bphi|}$.
Since $|\bh^{2r} + t\bphi| = \sqrt{1+t^2|\phi|^2}$, 
$\bn_t$ is well-defined. By calculation, we have
$$
\bn_t- \bh^{2r} = 
\frac{t\bphi}{\sqrt{1+t^2|\bphi|^2}} -
\frac{t^2|\bphi|^2 \bh^{2r}}{\sqrt{1+t^2|\bphi|^2} 
\left(1+\sqrt{1+t^2|\bphi|^2} \right)},
$$
$$
\rd_j \left(\bn_t- \bh^{2r} \right)= 
\frac{t\rd_j\bphi}{\sqrt{1+t^2|\bphi|^2}}
-
\frac{t^2|\bphi|^2 \rd_j\bh^{2r}}{\sqrt{1+t^2|\bphi|^2} 
\left(1+\sqrt{1+t^2|\bphi|^2} \right)}
-
\frac{t^3 (\bphi\cdot \rd_j\bphi) \bphi}{
(1+t^2|\bphi|^2)^{3/2}},
$$
which implies $\bn_t\to \bh^{2r}$ in $\boM_4$ as $t\to 0+$. 
Moreover, \eqref{e2.1} yields
$$
E_4 [\bn_t] - E_4[\bh^{2r}] =
\inp{\boL (\bn_t - \bh^{2r})}{\bn_t - \bh^{2r}} = t^2  \inp{\boL\bphi}{\bphi} + o(t^2),
$$
which is negative if $t$ is sufficiently small. This completes the proof.
\end{proof}

%
%

\section{Proof of Theorem \ref{T1}}\label{S3}

In this section, we show Theorem \ref{T1}. By Proposition \ref{P2.1}, it suffices to find $\bphi\in H^1(\R^2)$ with 
$$
\bphi \cdot \bh^{2r} =0,\qquad \text{and}\qquad \inp{\boL \bphi}{\bphi}<0.
$$
Following \cite{LiMel18}, one can rewrite the Hessian via several steps, using moving frame, Fourier expansion, and the Hardy decomposition. 

\subsection{Rescaling}

For $\bphi\in H^1(\R^2)$, we consider the $\dot{H}^1(\R^2)$ rescaling $\bphi^{2r} (\cdot) := \bphi \left( \frac{\cdot}{2r} \right)$. 
Then a simple calculation shows that the rescaled Hessian $\boH_r (\bphi) := 
\inp{\boL \bphi^{2r}}{\bphi^{2r}}$ can be written as 
\begin{equation}\label{e3.1}
\boH_r (\bphi) =
\nor{\nab \bphi}{L^2}^2
+ 4r^2 \inp{\nab \times \bphi}{\bphi}_{L^2}
+ 4r^2 \nor{\phi_3}{L^2}^2  
- \int_{\R^2} \La_r (\bh) |\bphi|^2 dx
\end{equation}
with
$$
\La_r (\bh) := |\nab \bh|^2 + 4r^2 \bh\cdot \nab \times \bh - 4r^2 (1-h_3) h_3.
$$
Note that the coefficients become balanced, and the dependence of the Hessian on $r$ gets more explicit. Our goal is now to find $\bphi\in H^1(\R^2)$ with
\begin{equation}\label{e3.15}
\bphi \cdot \bh =0,\qquad \text{and}\qquad \boH_r (\bphi) <0.
\end{equation}

\subsection{Moving frame}

Let $(\rho,\psi)$ be the polar coordinates in $\R^2$. Then by \eqref{e3}, we can write
\begin{equation}\label{e3.2}
\bh = 
\begin{pmatrix}
-\sin \psi \sin \te (\rho)\\
\cos \psi \sin \te(\rho)\\
\cos\te (\rho)
\end{pmatrix}
\end{equation}
where $\te:[0,\infty)\to \R$ is the non-decreasing function defined using
$$
\sin \te (\rho) = \frac{2\rho}{\rho^2+1},\qquad \te(0) = \pi, \qquad \te (\infty) =0.
$$
In particular, $\te$ satisfies the following relations:
\begin{equation}\label{e3.3}
\begin{aligned}
\cos \te = \frac{\rho^2-1}{\rho^2+1},\qquad && \sin\te -\rho (1-\cos \te) =0, \\
\te' = - \frac{2}{\rho^2+1} = -\frac{\sin \te}{\rho},\qquad && 
\te'' + \frac{\te'}{\rho} - \frac{\sin\te \cos\te}{\rho^2} =0. 
\end{aligned}
\end{equation}
Based on \eqref{e3.2}, we introduce the moving frame in the tangent space at $\bh^{2r}$ as
$$
\bJ_1 := 
\begin{pmatrix}
\cos \psi\\
\sin \psi\\
0
\end{pmatrix}
,\qquad
\bJ_2 := 
\begin{pmatrix}
-\sin \psi \cos \te (\rho)\\
\cos \psi \cos \te(\rho) \\
-\sin \te(\rho)
\end{pmatrix}
.
$$
For $\bphi \in T_{\bh} \S^2$, one can write
$$
\bphi = u_1 \bJ_1 + u_2 \bJ_2,\qquad \text{and then define } u= {}^t(u_1,u_2).
$$
Let us rewrite $\boH_r (\bphi)$ in terms of $u_1$ and $u_2$ following \cite{LiMel18}. 
First note that
$$
\begin{aligned}
\rd_\rho \bh = \te' \bJ_2,\qquad && \rd_\psi \bh = - (\sin \te) \bJ_1,\qquad 
&& \rd_\rho \bJ_1 =0, \\
\rd_\psi \bJ_1 = (\cos \te) \bJ_2 + (\sin \te) \bh,\qquad && 
\rd_\rho \bJ_2 = -\te' \bh,\qquad && 
\rd_\psi \bJ_2 = -(\cos \te) \bJ_1. 
\end{aligned}
$$
%
%
Hence, each component of the integrand in \eqref{e3.1} can be reexpressed as
$$
|\nab \bphi|^2 =
|\rd_\rho \bphi|^2 + \frac{1}{\rho^2} |\rd_\psi \bphi|^2 
=
 |\nab u|^2 + \frac{2\cos\te}{\rho^2} u\times\rd_\psi u + \frac{u_1^2}{\rho^2} 
+ \left((\te')^2 + \frac{\cos^2 \te}{\rho^2} \right)u_2^2,
$$
$$
\begin{aligned}
\bphi\cdot \nab \times\bphi 
&= 
\bphi \cdot 
\left[ \bJ_1\times \rd_\rho \bphi 
+ \frac{1}{\rho} \left( (\cos\te) \bJ_2 + (\sin\te) \bh  \right) \times \rd_\psi \bphi
\right] \\
&=
-\frac{\sin\te}{\rho} u\times \rd_\psi u +  \left( \te' - \frac{\sin\te\cos\te}{\rho} \right) u_2^2,
\end{aligned}
$$
$$
\phi_3^2 = u_2^2 \sin^2\te,\qquad 
\La_r (\bh) = (\te')^2 + \frac{\sin^2\te}{\rho^2} + 4r^2 \left( \te'+ \frac{\sin\te \cos\te}{\rho} \right) - 4r^2 (1-\cos\te)\cos\te.
$$
Hence by \eqref{e3.3}, we have
\begin{equation}\label{e3.5}
\begin{aligned}
\boH_r [\bphi] = \int_{\R^2}  |\nab u|^2 &+ \left( \frac{2\cos \te}{\rho^2} - \frac{4r^2\sin \te}{\rho} \right) 
u\times \rd_\psi u \\
&\hspace{20pt} + \left( 
- (\te')^2 + \frac{\cos^2 \te}{\rho^2} + 
\frac{ 4r^2\sin \te}{\rho}
\right) (u_1^2 + u_2^2) dx
\end{aligned}
\end{equation}

\subsection{Fourier splitting}

Next we apply Fourier expansion of $u_j$ with respect to $\psi$:
$$
u_j (\rho,\psi) = \al_j^{(0)} (\rho) + \sum_{k=1}^\infty \left( 
\al_j^{(k)} (\rho) \cos (k\psi) + \be_j^{(k)} (\rho) \sin (k\psi)
 \right),\qquad j=1,2.
$$
Then $\boH_r [\bphi]$ can be split in the following way:
\begin{equation}\label{e3.6}
\boH_r [\bphi] = 2\pi \boH_0^r [\al_1^{(0)},\al_2^{(0)}] 
+ \pi 
\sum_{k=1}^\infty \left( \boH_k^r  [\al_1^{(k)} , \be_2^{(k)}] + 
\boH_k^r  [\be_1^{(k)} , -\al_2^{(k)}] \right)
\end{equation}
where
\begin{equation}\label{e3.7}
\begin{aligned}
\boH_k^r [\al,\be] := \int_0^\infty &\left[
(\al')^2 + (\be')^2 +
 \left( \frac{k^2}{\rho^2}  - (\te')^2 + \frac{\cos^2 \te}{\rho^2} 
+ \frac{4r^2 \sin\te}{\rho}
\right) 
(\al^2+\be^2) \right. \\
&\left. \hspace{130pt} 
+ 4k \left( \frac{\cos \te}{\rho^2} - \frac{2r^2\sin \te}{\rho} \right) \al\be
\right] \rho d\rho.
\end{aligned}
\end{equation}
In order to find $\bphi\in H^1(\R^2)$ satisfying \eqref{e3.15}, it suffices to show that one of $\boH^r_k$ can take negative value. 
In fact, it can be shown that $\boH^r_0$, $\boH^r_1$ are always non-negative definite for all $r>0$. (See Appendix for the proof.) Thus we need to focus only on $\boH^r_k$ with $k\ge 2$.

\subsection{Instability at higher mode}

We show the following:
\begin{prop}[Instability at higher mode]\label{L2}
For $k\ge 2$, there exists $r_{k,c}\ge 1$ such that the following holds: If $r>r_{k,c}$, then 
there exist $\al,\be\in C_0^\infty (0,\infty)$ such that 
$\boH^r_k [\al,\be] <0$. Moreover, if $k=3$, then we can take $r_{3,c} =1$.
\end{prop}
For the proof, we change variables in the Hessian following the idea of \cite{LiMel18}. We use the following lemma:
\begin{lem}\label{La3}
Let $A:(0,\infty)\to\R$ be nonnegative $C^1$ function, let $V\in L^1_{\loc} ((0,\infty):\R)$, and let $L= -\frac{d}{d\rho} A(\rho) \frac{d}{d\rho} + V$. 
Let $f,g\in C_0^\infty(0,\infty)$ be functions satisfying $f= \psi g$ with 
some positive smooth function $\psi:(0,\infty)\to (0,\infty)$. Then, 
$$
\int_0^\infty (Lf)f d\rho = \int_0^\infty \psi^2 A (g')^2 + 
\int_0^\infty (L\psi) \psi g^2d\rho.
$$
\end{lem}
Lemma \ref{La3} plays a role of simplification of quadratic forms $\int_0^\infty (Lf)f d\rho$, especially when $L$ has a kernel as the ground state. 
Indeed, this is the case when $k=1$, and applying Lemma \ref{La3} immediately concludes that $\boH_1^r$ is positive definite (see the proof of Proposition \ref{P3} in Appendix). 
Although $\boH_k^r$ for $k\ge 2$ does not have such kernel, 
we will apply Lemma \ref{La3} with $\psi$ being the kernel of $\boH_1^r$, 
which enables us to find the unstable factors.\bigskip\par
%
%
%
\textit{Proof of Proposition \ref{L2}}. 
First, let us set $\al=\be$. 
Then we have
$$
\begin{aligned}
\boH_k^r [\al,\al] = & 2\int_0^\infty \left[
\rho (\al')^2 + \left(
\frac{(k+\cos\te)^2}{\rho}
- \rho (\te')^2 
+ 4r^2 (1-k) \sin \te
\right) \al^2
\right] d\rho \\
&= 2\int_0^\infty \left[
(L_1 \al) \al + 
\left(
\frac{k^2-1}{\rho} + \frac{2(k-1)\cos\te}{\rho} 
+4r^2 (1-k) \sin \te
\right)
\right] d\rho
\end{aligned}
$$
where $L_1:= -\frac{d}{d\rho} \rho \frac{d}{d\rho} 
+ \frac{(1+\cos\te)^2}{\rho} + 4r^2 (1-k) \sin\te$. 
Noting that $L_1 \left( \frac{\sin\te}{\rho} \right)=0$, 
we transform 
$\al=\frac{\sin \te}{\rho} \xi$. 
Applying Lemma \ref{La3} with $A=\rho$, $\psi=\frac{\sin\te}{\rho}$, 
$V= \frac{(1+\cos\te)^2}{\rho} - \rho (\te')^2$, we have
$$
\begin{aligned}
\boH_k^r \left[\frac{\sin \te}{\rho} \xi, \frac{\sin \te}{\rho} \xi \right] 
=
\int_0^\infty 
&\left[ \frac{2\sin^2 \te}{\rho} (\xi')^2 
+ 
f^r_k (\rho) \xi^2  \right] d\rho
%
\end{aligned}
$$
where
$$
f_k^r(\rho):= 2(k^2-1) \frac{\sin^2 \te}{\rho^3} 
+ 4(k-1) \frac{\sin^2\te\cos\te}{\rho^3}
+ 8 (1-k) r^2 \frac{\sin^3\te}{\rho^2}. 
$$
Since
\begin{align*}
\sin \te = \frac{2}{\rho} + o(\rho^{-1}), &&
\cos \te = 1 + o(\rho^{-1})
\end{align*}
as $\rho\to\infty$, we have 
\begin{align*}
f_k^r (\rho) = - 8 (k-1) (8r^2 - k -3 ) \rho^{-5} + o(\rho^{-5}).
\end{align*}
Now we consider the rescaling
$$
\xi_\la (\rho) := \frac{1}{\la^2} \xi(\la\rho),\qquad \la>0.
$$
Then for $\xi\in C_0^\infty(0,\infty)$, we have
$$
\begin{aligned}
\boH_k^r \left[\frac{\sin \te}{\rho} \xi_\la, \frac{\sin \te}{\rho} \xi_\la \right]= 
\int_0^\infty 
 \left[\frac{8}{\rho^3} (\xi')^2 - \frac{8 (k-1) (8r^2 - k -3 )}{\rho^5} \xi^2 \right] d\rho + o(\la). 
\end{aligned}
$$
as $\la\to 0+$. Thus we obtain
\begin{equation*}
\lim_{\la\to 0+} 
\boH_k^r \left[\frac{\sin \te}{\rho} \xi_\la, \frac{\sin \te}{\rho} \xi_\la \right]
=
\int_0^\infty \left[ \frac{8}{\rho^3} (\xi')^2 - \frac{8 (k-1) (8r^2 - k -3 )}{\rho^5} \xi^2
\right] d\rho,
\end{equation*}
which we denote $\boI^r_k [\xi]$. Hence for Lemma \ref{L2}, it suffices to show that $\boI^r_k[\xi] <0$ for some $\xi\in C_0^\infty (0,\infty)$. 
This problem is concerned with the optimization of the constant of the Hardy-type inequality:
\begin{equation}\label{e3.8}
C_{H} := \inf \left\{ 
C \ \left|\ 
\int_0^\infty \frac{\xi^2}{\rho^5} d\rho \le C \int_0^\infty \frac{(\xi')^2}{\rho^3}d\rho \quad
\text{for all } \xi\in C_0^\infty (0,\infty).\right.
\right\}.
\end{equation}
%
In fact, it is known that $C_H=\frac 14$, and thus for any $\eps>0$, there exists $\xi_\eps\in C_0^\infty (0,\infty) \setminus \{0\}$ such that
$$
\int_0^\infty \frac{\xi_\eps^2}{\rho^5} d\rho > \left( \frac 1{4 +\eps} \right) 
\int_0^\infty \frac{(\xi'_\eps)^2}{\rho^3}d\rho.
$$
This fact is shown in \cite{HLP} (see also \cite{PerSam15}), 
while in Appendix we will reproduce the proof for reader's convenience. 
Using $\xi_\eps$, we have
$$
\begin{aligned}
\boI_k^r [\xi_\eps] 
&< \int_0^\infty 
\left[\frac{8(4+\eps)}{\rho^5} \xi_\eps^2 - \frac{8 (k-1) (8r^2 - k -3 )}{\rho^5} \xi^2_\eps 
\right] d\rho\\
&= 8[4+\eps-(k-1)(8r^2-k-3)] \int_0^\infty \frac{1}{\rho^5} \xi_\eps^2 d\rho.
\end{aligned}
$$
If $k\ge 2$, then the right hand side is negative for sufficiently large $r$. 
Especially when $k=3$, it holds that
$$
\boI_3^r [\xi_\eps] < 128 \left(1-r^2 +\frac{\eps}{16}\right) \int_0^\infty \frac{1}{\rho^5} \xi_\eps^2 d\rho
$$
which is negative when $r>1$ if $\eps$ is sufficiently small. Thus the proof of Lemma \ref{L2} is complete.\qedhere


\textit{Proof of Theorem \ref{T1}}. 
According to Lemma \ref{L2}, 
if $r>1$, then there exist $\al,\be\in C_0^\infty(0,\infty)$ such that 
$\boH^r_3 [\al,\be] <0$. Now define
\begin{equation}\label{e3.9}
u_1(\rho,\psi) := \al (\rho) \cos (3\psi),\qquad u_2(\rho,\psi) := -\be (\rho) \sin(3\psi).
\end{equation}
Then by \eqref{e3.5} and \eqref{e3.6}, 
$\bphi := u_1 \bJ_1 + u_2 \bJ_2$ satisfies $\boH_r[\bphi] <0$, which completes the proof by Proposition \ref{P2.1}.

\begin{remark}
Note that $\al, \be$ is taken such that $\al=\be =\frac{\sin\te}{\rho}\xi $ for a specific $\xi\in C_0^\infty (0,\infty)$. 
Thus $\bphi$ in the proof can be written in the form
$$
\bphi = \frac{\sin\te(\rho)}{\rho}\xi (\rho) \cos (3\psi) \bJ_1 - 
\frac{\sin\te(\rho)}{\rho}\xi(\rho) \sin (3\psi) \bJ_2 
= \xi(\rho) \left( -\frac{\cos (3\psi)}{\rho} \rd_\psi \bh 
+\sin (3\psi) \rd_\rho \bh
 \right).
$$
\end{remark}

%
%

\section{Proof of Theorem \ref{T2}}\label{S4}

Let $r>1$. In this section we construct 
a sequence $\{\bn_\nu\}_{\nu=1}^\infty$ with
\begin{align}\label{e4.1}
\bn_\nu \in \boM_4, && 
Q[\bn_\nu] = -1, &&
\text{and}\quad \lim_{n\to\infty} E_4 [\bn_\nu] = -\infty.
\end{align}
The key ingredient is the specific map defined as
$$
\bb (x) := 
\left(
0,\frac{2rx_1}{r^2(x_1)^2+1},
\frac{r^2(x_1)^2-1}{r^2(x_1)^2+1}
\right).
$$
Note that $\bb (x) = \bh^{1/r} (x_1,0)$ where $\bh$ is as in \eqref{e3}, and $\bb$ satisfies an equation similar to the Beltrami field:
$$
\nab \times \bb = -\frac{b_2}{x_1} \bb.
$$
If we calculate the integrand of $E_4$, then
\begin{equation}\label{e4.2}
\frac 12 |\nab \bb|^2 + r (\bb -\boe_3) \cdot \nab \times \bb +\frac 12 (b_3-1)^2 
= \frac{2(1-r^2)}{(r^2x_1^2+1)^2}.
\end{equation}
Thus the energy of $\bb$ is $-\infty$ if $r>1$. 
Our construction of $\{\bn_\nu\}$ with \eqref{e4.1} is based on the cut-off of $\bb$. \par
%
Define
$$
\bn_L (x) := 
\left\{
\begin{aligned}
&\bb (x) & \text{if } |x_2|\le L,\\
&\bh^{1/r} (x_1, x_2-L) & \text{if } x_2 > L,\\
&\bh^{1/r} (x_1, x_2+L) & \text{if } x_2 < -L.
\end{aligned}
\right.
$$
Then $\bn_L\in \boM_4\cap C(\R^2)$, and we have $Q[\bn_L]=-1$ since $\bn_L$ is homotopic to $\bh^{r}$ by the homotopy with $L$ shrinking to $0$. 
Moreover, \eqref{e4.2} gives
$$
\begin{aligned}
E_4 [\bn_L] 
&=
\int_{\{ |x_2|\le L\}} 
\left(\frac 12 |\nab \bb|^2 + r (\bb -\boe_3) \cdot \nab \times \bb +\frac 12 (b_3-1)^2 \right)
dx\\
&\hspace{20pt}
+ 
\int_{\{ x_2> L\}} 
\left[
\frac 12 |\nab \bh^{1/r}|^2 + r (\bh^{1/r} -\boe_3) \cdot \nab \times \bh^{1/r} +\frac 12 (h^{1/r}_3-1)^2 \right](x_1,x_2-L) dx \\
&\hspace{20pt} 
+ 
\int_{\{ x_2< -L\}} 
\left[
\frac 12 |\nab \bh^{1/r}|^2 + r (\bh^{1/r} -\boe_3) \cdot \nab \times \bh^{1/r} +\frac 12 (h^{1/r}_3-1)^2 \right] (x_1, x_2+L)dx\\
&= 
\int_{\{ |x_2|\le L\}} \frac{2(1-r^2)}{(r^2x_1^2+1)^2} dx
+ E_4[\bh^{1/r}]  = (1-r^2) C_r L +E_4 [\bh^{1/r}]
\end{aligned}
$$
where $C_r$ is positive constant independent of $L$. 
Hence we have
$$
\lim_{L\to\infty} E_4[\bn_L] =-\infty
$$
which concludes the proof of Theorem \ref{T2}.


%
%

\section{Appendix}\label{S5}

\subsection{Positivity of the Hessian at lower modes}

We show that $\boH^r_k$ defined as \eqref{e3.7} is positive definite. 

\begin{prop}\label{P3}
For all $r\ge 0$, 
we have $\boH_0^r[\al,\be]\ge 0$, $\boH_1^r[\al,\be]\ge0$ for all $\al$, $\be\in C_0^\infty (0,\infty)$.
\end{prop}

\begin{proof}
Our proof essentially follows \cite{LiMel18}. 
We first see the case $k=0$:
$$
\boH_0^r [\al,\be]
=
\int_0^\infty \left[ 
(\al')^2 + (\be')^2 + 
\left( -(\te')^2 + \frac{\cos^2 \te}{\rho^2} \right) ( \al^2 + \be^2)
+4r^2 \frac{\sin \te}{\rho} (\al^2+\be^2) 
\right]
\rho d\rho.
$$
Clearly it suffices to show the case $r=0$. Now we can write
$$
\boH_0^0 [\al,\be]
= \int_0^\infty \left[ (L_0 \al ) \al + (L_0 \be ) \be\right] d\rho
$$
with $L_0 := -\frac{d}{d\rho} \rho \frac{d}{d\rho}  
-\rho (\te')^2 + \frac{\cos^2 \te}{\rho}
$. 
Noting that $L_0(\sin \te) =0$, we transform
$$
\al = (\sin \te) \xi,\qquad \be = (\sin \te) \eta,\qquad (\xi,\eta\in C_0^\infty(0,\infty)).
$$
Applying Lemma \ref{La3} with $A=\rho$, $\psi= \sin \te$, $V= -\rho(\te')^2 + \frac{\cos^2 \te}{\rho}$, we obtain
$$
\boH_0^0[(\sin \te)\xi, (\sin \te) \eta] = \int_0^\infty \sin^2 \te [(\xi')^2 + (\eta')^2] \rho d\rho \ge 0.
$$
Next we consider the case $k=1$. We can write
\begin{align*}
\boH_1^r [\al,\be] = &
\int_0^\infty 
\left[ (\al')^2 + (\be')^2 + \left( \frac{1}{\rho^2} - (\te')^2 + \frac{\cos^2 \te}{\rho^2} \right) (\al^2+\be^2) + \frac{4\cos \te}{\rho^2} \al\be
\right.\\
&\left. 
\hspace{150pt}
+ \frac{4r^2 \sin\te}{\rho} (\al-\be)^2 \right] \rho d\rho.
\end{align*}
Thus it also suffices to show the case $r=0$. Then we can write
$$
\boH_1^0 [\al,\be] 
= 
\int_0^\infty 
\left[ (L_1 \al ) \al + (L_1 \be ) \be 
- \frac{2 \cos\te}{\rho} (\al-\be)^2 
\right] d\rho
$$
with 
$L_1 := -\frac{d}{d\rho} \rho \frac{d}{d\rho}  
+ \frac{(1 + \cos \te)^2}{\rho} - \rho (\te')^2$. 
Noting that $L_1 (\frac{\sin\te}{\rho}) =0$, we transform
$$
\al = \frac{\sin \te}{\rho} \xi,\qquad \be = \frac{\sin \te}{\rho} \eta,\qquad (\xi,\eta\in C_0^\infty(0,\infty)).
$$
Applying Lemma \ref{La3} with $A=\rho$, $\psi=\frac{\sin\te}{\rho}$, $V= \frac{(1+\cos\te)^2}{\rho} - \rho (\te')^2$, and using \eqref{e3.3}, we have
$$
\begin{aligned}
\boH_1^0 
\left[\frac{\sin \te}{\rho} \xi, \frac{\sin \te}{\rho} \eta\right]
&=
\int_0^\infty \frac{\sin^2\te}{\rho} (\xi'^2+\eta'^2) 
+ 
\frac{2\sin\te\cos\te (\te')}{\rho^2} (\xi-\eta)^2 d\rho \\
&=
\int_0^\infty \frac{\sin^2\te}{\rho} (\xi'^2+\eta'^2) 
+ 
\frac{(\sin^2\te)'}{\rho^2} (\xi-\eta)^2 d\rho \\
&= 
\int_0^\infty \frac{\sin^2\te}{\rho} (\xi'^2+\eta'^2) 
- 
\sin^2\te 
\left(
\frac{2 (\xi-\eta)(\xi'-\eta')}{\rho^2}
- \frac{2(\xi-\eta)^2}{\rho^3}
\right)
d\rho\\
&=
\int_0^\infty
\frac{\sin^2 \te}{\rho}
\left[
\left(\xi'-\frac{\xi-\eta}{\rho}\right)^2 
+ \left( \eta' + \frac{\xi-\eta}{\rho} \right)^2
\right]\ge 0.
\end{aligned}
$$
Hence the proof is complete.
\end{proof}

\subsection{The optimality of the Hardy-type inequality}

In this section we give a proof of the optimality $C_H = \frac 14$. 
(For the proof of $C_H >\frac 14$, see \cite{HLP}.) More precisely, we show the following:

\begin{lem}
For any $\eps>0$ there exists $\xi_\eps\in C_0^\infty (0,\infty) \setminus \{0\}$ such that
$$
\left( \frac 1{4 +\eps} \right) 
\int_0^\infty \frac{(\xi'_\eps)^2}{\rho^3} d\rho
< \int_0^\infty \frac{\xi_\eps^2}{\rho^5} d\rho .
$$
\end{lem}

\begin{proof}
As given in \cite{HLP}, $C_H$ is formally optimized by $\xi= \rho^{2}$. To seek 
the compactness of support, 
we take a cut-off of this function. 
Given $A>1$, 
let $\chi=\chi_A\in C_0^\infty(0,\infty)$ be a function with
$\chi (\rho ) = 1$ if $\rho\in [1,A]$, $\chi(\rho)=0$ if $\rho \not\in [\frac 12 ,2A]$, 
$|\chi'(\rho)|\le \frac 2A$ for $\rho\in [A,2A]$, and 
$0\le \chi(\rho) \le 1$ for all $\rho\in (0,\infty)$.
Then define
$$
\xi_{A} (\rho) := \rho^{2} \chi_A (\rho)
$$
for $\eps>0$. Then calculation gives
$$
\int_0^\infty \frac{\xi_A^2}{\rho^5} d\rho 
= 
\int_0^\infty \frac{1}{\rho} 
\chi_A^2(\rho) d\rho,
$$
$$
\begin{aligned}
\int_0^\infty \frac{(\xi'_A)^2}{\rho^3} d\rho 
&=
4 
\int_0^\infty 
\frac{1}{\rho} \chi_A^2 d\rho
+ 
4 \int_0^\infty \chi_A \chi_A' d\rho 
+
\int_0^\infty \rho (\chi_A')^2 d\rho \\
&=
4 \int_0^\infty 
\frac{1}{\rho} \chi_A^2 d\rho
+ \int_0^\infty \rho (\chi_A')^2 d\rho.
\end{aligned}
$$
Thus it suffices to show that given $\eps>0$, there exists $A>0$ such that
\begin{equation}\label{e5.1}
\int_0^\infty \rho (\chi_A')^2 d\rho \le 
\eps \int_0^\infty 
\frac{1}{\rho} \chi_A^2 d\rho.
\end{equation}
For \eqref{e5.1}, 
we can estimate as
$$
\int_0^\infty 
\frac{1}{\rho} \chi_A^2 d\rho \ge \int_1^A \frac{1}{\rho} = 
\log A,
$$
$$
\int_0^\infty \rho (\chi_A')^2 d\rho 
\le C \int_{\frac 12}^1 \rho d\rho + \frac{4}{A^2} \int_A^{2A} \rho d\rho \le C 
$$
where $C$ is independent of $A$. Thus we have
$$
\left[ \int_0^\infty 
\frac{1}{\rho} \chi_A^2 d\rho \right]^{-1}
\int_0^\infty \rho (\chi_A')^2 d\rho 
\xrightarrow{A\to\infty} 0
$$
which implies \eqref{e5.1} for sufficiently large $A$.
\end{proof}

\textbf{Acknowledgements} S. Ibrahim is supported by the NSERC grant No. 371637-2019. I. Shimizu is supported by JSPS KAKENHI Grant Number 19H05599. 


\noindent(S. Ibrahim): ibrahims@uvic.ca\\
Department of Mathematics and Statistics,\\ University of Victoria, Victoria, BC, Canada\\
\vspace{10pt}
\\
\noindent(I. Shimizu): i.shimizu@sigmath.es.osaka-u.ac.jp
\\Mathematical Science, Graduate School of Engineering Science,\\
Osaka University, Toyonaka, Osaka, Japan 560-0043.

\end{document}